\documentclass[production,11pt]{jsg}
\usepackage{amscd}
\usepackage[mathscr]{eucal}
\usepackage{graphicx,amssymb}
\usepackage{amsxtra}
\usepackage{amsmath}
\usepackage[all]{xy}

\graphicspath{{H://IP/JSG/ISSUE/8-1/Pagination/art/}}

\pubyr{2010}

\vol{8}

\iss{1}

\startpage{1}

\endpage{6}

\newtheorem{thm}{Theorem}[section]

\theoremstyle{remark}

\newtheorem{rem}{Remark}[section]

\begin{document}

\title[A maximal relative symplectic packing construction]{A maximal relative symplectic packing construction}

\author[L. Buhovsky]{Lev Buhovsky\footnote{The author was partially supported by the ISRAEL SCIENCE
FOUNDATION (grant No.~1227/06 *).}}

\address{The Mathematical Sciences Research Institute\\ Berkeley, CA 94720-5070\\ USA}
\email{levbuh@gmail.com}

\urladdr{\newline
Received 12/12/2008, accepted 8/17/2009}

\maketitle

\begin{abstract}
 In this paper we present an explicit construction of a relative symplectic packing.
 This confirms the sharpness of the upper bound for the relative
 packing of a ball into the pair $ ( \mathbb{CP}^2, \mathbb{T}_{ \rm{Cliff}}^2 )
 $ of the standard complex projective plane and the Clifford
 torus, obtained by Biran and Cornea.
\end{abstract}

\section{Introduction and main results} \label{S:intro}

In this note we present an explicit construction of a relative
packing. The subject of symplectic packing was introduced first in
the seminal work of Gromov \cite{Gr}. Gromov showed that looking
at symplectic embeddings of a standard ball into a symplectic
manifold, one may obtain an upper bound on the radius of a ball
which is stronger than the obstruction coming from the volume. The
first theorem in this direction is a non-squeezing theorem from
\cite{Gr}. This result has led to the definition of the Gromov
capacity, which plays an important role in modern symplectic
geometry.
 Later the subject of symplectic packing was treated by
Biran, Karshon, Mc'Duff, Polterovich, Schlenk, Traynor and others,
see \cite{Bi-1,Bi-2,Bi-3,Bi-4,K,M-P,Sch-1,Sch-2,Sch-3,Tr}. New
obstructions for symplectic packings of various domains were
found. On the other hand, attempts were made to find explicit
constructions of certain symplectic embeddings, in order to show
that the obstructions, which were found, are tight (see, e.g.,
\cite{K,Sch-4,Tr}). This note is devoted to proving a new result
in this direction.

 Recently, Biran and Cornea \cite{Bi-Co} found new
obstructions on the relative symplectic packing in a number of
situations, which are stronger than those in the case of a usual
packing. In this note we consider one specific example
from \cite{Bi-Co}, and show that the corresponding obstruction is
sharp.

Let us define first the notion of a relative symplectic packing.
Consider a symplectic manifold $ (M^{2n},\omega) $ and a closed
Lagrangian submanifold $ L^{n} \subset M $. Take a standard open
ball $ B^{2n}(r) \subset (\mathbb{R}^{2n} , \omega_{\rm std} ) $
of radius $ r>0 $, where $ \mathbb{R}^{2n} $ is endowed with
coordinates $ (q,p) $, and set \[ B^{2n}_{\mathbb{R}}(r) = \{
(q,p) \in B^{2n}(r) \mid p=0 \} .\] A relative packing of $
B^{2n}(r) $ into $ (M,L) $ is by definition a symplectic embedding
\[
i : (B^{2n}(r), \omega_{\rm std} ) \hookrightarrow (M^{2n},\omega
)
\]
such that $ i^{-1}(L)=
B^{2n}_{\mathbb{R}}(r) $.

 In this note we treat the situation where $ (M, \omega) = (
\mathbb{CP}^{2},\omega_{\rm FS})$ and $ L = \mathbb{T}^{2}
\hookrightarrow ( \mathbb{CP}^{2},\omega_{\rm FS})$ is the
standard Clifford torus in $ M $. Here, $ \omega_{\rm FS} $ is
normalized so that $ \int_{\mathbb{CP}^1} \omega_{\rm FS} = \pi$.
In \cite{Bi-Co} it was shown that given a relative packing
\[
\big( B^{4}(r) , B^{4}_{\mathbb{R}}(r) \big) \hookrightarrow (
\mathbb{CP}^{2},\mathbb{T}^{2}),
\]
there is an upper bound for
the radius of this ball : $ r \leqslant \sqrt{\frac{2}{3}} $.
 Our main result is the following.

\begin{thm} \label{T:pack example}
 For every $ r < \sqrt{\frac{2}{3}} $ there exists a relative packing
 \[
 \big( B^{4}(r) , B^{4}_{\mathbb{R}}(r) \big) \hookrightarrow ( \mathbb{CP}^{2},\mathbb{T}^{2}).
 \]
\end{thm}

 Let us mention that in \cite{Bi-Co}, the authors also consider relative packings into
$ ( \mathbb{CP}^{2},\mathbb{T}^{2}) $ by more than one ball,
and obtain obstructions on their radii. The case of three balls
was treated, and is strongly connected to the properties of the
quantum cup-product in Floer homology. The hypothetical upper
bound in this case was found, under the assumption of existence of
pseudo-holomorphic discs with certain properties. It still remains
to show the existence of such discs, and in the case that it will
be proved, one can try to find the example which proves the
tightness of this upper bound.

 The rest of the paper is devoted to proving Theorem~\ref{T:pack
example}. In Section~\ref{S:ConstrOver} we show how to reduce
Theorem ~\ref{T:pack example} to a two-dimensional problem.
Section ~\ref{S:Proofs} contains the proof of Theorem ~\ref{T:pack
example}.

\section{Overview of the construction} \label{S:ConstrOver}

 It is well-known that the open symplectic manifold $ ( \mathbb{CP}^{2} \setminus
\mathbb{CP}^{1}, \omega_{\rm FS} ) $ is symplectomorphic to the
unit ball $ ( B^{4}(1), \omega_{\rm std} ) \subset (
\mathbb{R}^{4}, \omega_{\rm std} ) $. Identifying $ \mathbb{R}^{4}
\cong \mathbb{C}^{2} $, we have a natural action of the torus $
\mathbb{T}^{2} $ on $ B^{4}(1) $. The moment map of this action is
given by
\begin{gather*}
B^{4}(1) \rightarrow \mathbb{R}^{2},\\
(z_1,z_2) \mapsto ( |z_1|^{2} , |z_2|^{2} ).
\end{gather*}

The action of $ \mathbb{T}^2 $, restricted to the complement $ B^4(1) \setminus \{ z_1 z_2 = 0 \} $ of the union
of the two complex axes is free, and therefore, by a standard procedure, we obtain that
$ B^4(1) \setminus \{ z_1 z_2 = 0 \} $ is symplectomorphic to $ \mathbb{T}^2 \times \triangle $. Here we use the notation
$$ \triangle = \{ (p_1,p_2) \mid p_1,p_2 >0 , p_1 + p_2 <1 \} \subset \mathbb{R}^{2} .$$

 Look now at $$ K := \square \times \mathbb{R}^{2}(p_{1},p_{2}) \subset \mathbb{T}^{2} \times \mathbb{R}^{2} ,$$
and its subset $ K' := \square \times \triangle \subset K $, where
$$ \square = \{ (q_1,q_2) \mid 0 < q_1,q_2 < \pi \} \subset \mathbb{T}^{2} .$$
The above symplectomorphism between $ B^4(1) \setminus \{ z_1 z_2 = 0 \} $ and $ \mathbb{T}^2 \times \triangle $
induces a symplectic embedding $$ j: K' \hookrightarrow B^4(1) .$$
 From now on we will consider $ K,K' $
as $$  K = \{ ( q_{1},p_{1} ) \mid 0<q_{1}<\pi  \} \times \{ ( q_{2},p_{2} ) \mid 0<q_{2}< \pi \} \subset \mathbb{R}^{4},$$
$$  K' = \{ ( q_{1},p_{1},q_{2},p_{2} ) \mid 0<q_{1}<\pi , 0<q_{2}< \pi , p_1,p_2 >0 , p_1 + p_2 <1 \} \subset K ,$$
and the symplectic form is $$ dp_{1} \wedge dq_{1} + dp_{2}
\wedge dq_{2} .$$ The Clifford torus lies entirely in $ B^4(1) $, and its pre-image
under the map $ j $ equals to  $$ L' = \{ ( q_{1},p_{1},q_{2},p_{2} ) \mid 0 < q_{1} < \pi, 0 < q_{2} <
\pi, p_{1}=p_{2}=1/3 \} \subset K' .$$ Fix $ r < \sqrt{\frac{2}{3}} $ and consider
\[
B^{4}(r) \subset B^{2}(r)
\times B^{2}(r) \subset \mathbb{R}^{4} .
\]
The construction is
based on finding a certain area-preserving map
\[
\sigma :
B^{2}(r) \rightarrow \mathbb{R}^2 .
\]
Given such an $ \sigma $, we
define $ \Phi : B^{2}(r) \times B^{2}(r) \rightarrow \mathbb{R}^{4} $ as
\[
\Phi (z,w)= \big( \sigma (z),\sigma (w) \big) .
\]
Our aim is to find $ \sigma $ such that
the image of $ B^{4}(r) \hookrightarrow B^{2}(r) \times B^{2}(r) $
under the map $ \Phi $ will be contained in $ K' $, and also
\[
\Phi^{-1}( L' ) \cap B^{4}(r) = B^{4}_{\mathbb{R}}(r) .
\]

\section{Proofs} \label{S:Proofs}

\begin{proof}[Proof of Theorem~\ref{T:pack example}]
Given $ r < \sqrt{\frac{2}{3}} $, we describe a construction of a relative symplectic embedding
 \[
 \big( B^{4}(r) , B^{4}_{\mathbb{R}}(r) \big) \hookrightarrow (\mathbb{CP}^{2},\mathbb{T}^{2}) .
 \]
As it was shown in Section ~\ref{S:ConstrOver}, it is enough to find a symplectic embedding
\[
\Phi: B^{4}(r) \rightarrow \mathbb{R}^{4} ,
\] such that its image
will be contained in the domain
\[
\{ ( q_{1},p_{1},q_{2},p_{2} ) \mid 0<q_{1}< \pi ,
 0<q_{2}< \pi , 0< p_{1} , 0< p_{2} , 0 < p_{1} + p_{2} < 1 \},
\]

\noindent and the pre-image of
 \[
 \{ ( q_{1},p_{1},q_{2},p_{2} ) \mid 0<q_{1}< \pi , 0<q_{2}< \pi,
p_{1}=p_{2}=1/3 \}
\]
will be equal to $B^{4}_{\mathbb{R}}(r) \subset B^{4}(r)$.

\begin{figure}[!b]
\centering{\includegraphics{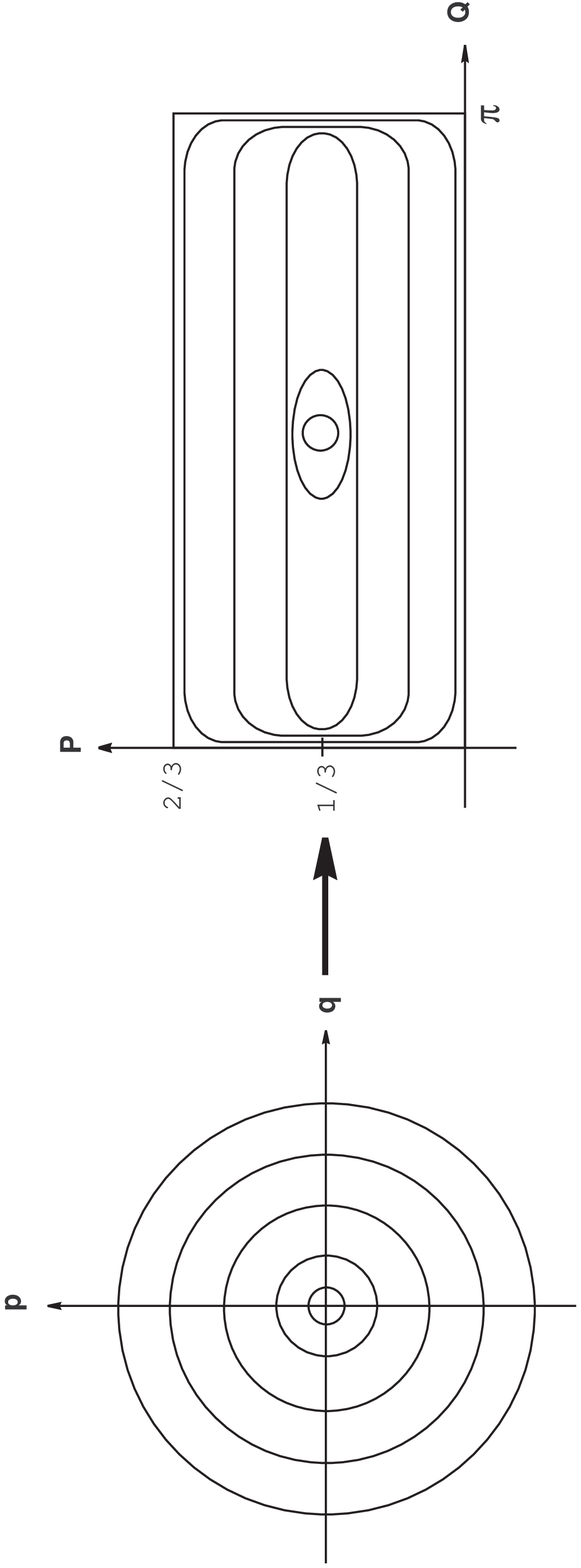}}
\caption{The map $\sigma$\label{fig1}}
\end{figure}

As it was shown by Schlenk (\cite{Sch-4}, Lemma 3.1.5), for
any $\epsilon>0$ there exists an area-preserving
diffeomorphism
\[
\sigma : B^{2}(r) \rightarrow (0,\pi) \times \left(0,\frac{2}{3}\right) \subset \mathbb{R}^{2}(Q,P)
\]
with the following properties (Figure~1):
\begin{enumerate}
\item[(1)]
For each $ u \in (o, r^2 ) $ one has: if $ p^2 \,{+}\,
q^2 \leqslant u $, then $ P(\sigma(q,p)) \leqslant
\frac{1}{3} + \frac{u}{2} + \epsilon$.
\item[(2)]
The line $ \{ p=0 \} $ in $ B^{2}(r)$, and no other
points of $ B^{2}(r) $, is mapped to the line $ \{
(Q,\frac{1}{3})\ |\ Q \in (0,\pi) \}$.
\end{enumerate}

Take $ \epsilon = \frac{1}{2} \big(\frac{1}{3} - \frac{r^2}{2} \big) $, and consider the corresponding map $ \sigma $.
Then for any $ z,w \in \mathbb{R}^{2} $, such that $ (z,w) \in B^{4}(r) $, define
\[
\Phi (z,w)= \big( \sigma (z),\sigma (w) \big) .
\]
We claim that the resulting map $ \Phi : B^{4}(r) \rightarrow
\mathbb{R}^{4} $ satisfies the desired properties. First of all,
it is clear that the map $ \Phi $ is a symplectic embedding, and
the pre-image of
\[
\{ ( q_{1},p_{1},q_{2},p_{2} ) \mid 0<q_{1}< \pi , 0<q_{2}< \pi, p_{1}=p_{2}=1/3 \}
\]
equals $B^{4}_{\mathbb{R}}(r) \subset B^{4}(r) $. Let us show
that, moreover, the image of $ \Phi $ lies in the domain
\[
\{ ( q_{1},p_{1},q_{2},p_{2} ) \mid 0<q_{1}< \pi ,
 0<q_{2}< \pi , 0< p_{1} , 0< p_{2} , 0 < p_{1} + p_{2} < 1 \} .
\]

Take any $ (q_1,p_1,q_2,p_2) \in B^{4}(r) $, then we have
\[
q_{1}^2 + p_{1}^2 + q_{2}^2 + p_{2}^2 < r^2 .
\]
Set $ u = q_{1}^2 + p_{1}^2 $. Then $ q_{2}^2 + p_{2}^2 < r^2 - u $, and so
\begin{align*}
P_{1}(\sigma (q_1,p_1)) + P_{2}(\sigma (q_2,p_2)) &<
\left(\frac{1}{3} + \frac{u}{2} + \epsilon\right) +
\left(\frac{1}{3} + \frac{r^2 - u}{2} + \epsilon\right)\\[6pt]
&= \frac{2}{3} + \frac{r^2}{2} + 2\epsilon \\
&= 1.\\[-3pc]
\end{align*}
\end{proof}

\begin{rem}
 The presented relative packing construction can be naturally generalized to the
corresponding construction of a relative packing of a
\hbox{$2n$-dimensional} ball $ B^{2n}(r) $ into $ ( \mathbb{CP}^n,
\mathbb{T}_{\rm{Cliff}}^n ) $, for any $ n \geqslant 2 $ and radius $ r < \sqrt{\frac{2}{n+1}}$.
This confirms the sharpness of the upper bound for the
radius, obtained by Biran and Cornea~\cite{Bi-Co}, for an
arbitrary dimension.
\end{rem}

\subsubsection*{Acknowledgments} I would like to thank Paul Biran for a valuable
advice. I would also like to thank the referee for many useful
remarks and comments helping to improve the exposition.


\begin{thebibliography}{999999}

\bibitem[Bi-1]{Bi-1} P. Biran, \textit{Symplectic packing in dimension $4$}, Geom. Funct. Anal.
\textbf{7}(3) (1997), 420--437.

\bibitem[Bi-2]{Bi-2} P. Biran, \textit{A stability property of symplectic packing}, Invent. Math.
\textbf{136} (1999), 123--55.

\bibitem[Bi-3]{Bi-3} P. Biran, \textit{From symplectic packing to algebraic geometry and
back}, Proceedings of the 3'rd European Congress of Mathematics
(Barcelona 2000), Vol. II, 507-524, Progr. Math.,
\textbf{202} , Birkhauser, Basel, 2001.

\bibitem[Bi-4]{Bi-4} P. Biran, \textit{Lagrangian barriers and symplectic
embeddings}, Geome. Funct. Anal. \textbf{11}(3) (2001), 407--464.

\bibitem[Bi-Co]{Bi-Co} P. Biran and O. Cornea, \textit{Quantum structures for
Lagrangian submanifolds}, preprint, http://arxiv.org/abs/0708.4221

\bibitem[Gr]{Gr} M. Gromov, \textit{Pseudoholomorphic curves in
symplectic manifolds}, Invent. Math. \textbf{82}(2) (1985), 307--347.

\bibitem[K]{K} Y. Karshon, \textit{Appendix (to D. McDuff and L.
Polterovich)}, Invent. Math. \textbf{115}(1) (1994), 431--434.

\bibitem[M-P]{M-P} D. McDuff and L. Polterovich, \textit{Symplectic packings and
algebraic geometry. With an appendix by Yael Karshon},
Invent. Math. \textbf{115}(3) (1994), 405--434, 307--347.

\bibitem[Sch-1]{Sch-1} F. Schlenk, \textit{An extension theorem in symplectic
geometry}, Manuscripta Math. \textbf{109} (2002), 329--348.

\bibitem[Sch-2]{Sch-2} F. Schlenk, \textit{Volume preserving embeddings of open subsets of
$R^n$ into manifolds}, Proc. Amer. Math. Soc. \textbf{131} (2003), 1925--1929.

\bibitem[Sch-3]{Sch-3} F. Schlenk, \textit{Symplectic embeddings of
ellipsoids}, Israel J. Math. \textbf{138} (2003),\break 215--252.


\bibitem[Sch-4]{Sch-4} F. Schlenk, \textit{Embedding problems in symplectic
geometry}, De Gruyter Expositions in Mathematics, \textbf{40}, Walter de Gruyter Verlag,
Berlin, 2005.

\bibitem[Tr]{Tr} L. Traynor, \textit{Symplectic packing
constructions}, J. Diff. Geom. \textbf{42}(2) (1995),\break 411--429.

\end{thebibliography}
\end{document}